\documentclass[letterpaper,11pt]{amsart}

\usepackage[margin=1.2in]{geometry}
\usepackage{amsmath,amsthm,amssymb}
\usepackage{xspace,xcolor}
\usepackage[breaklinks,colorlinks,citecolor=teal,linkcolor=teal,urlcolor=teal,pagebackref,hyperindex]{hyperref}
\usepackage[alphabetic]{amsrefs}
\usepackage[all]{xy}
\usepackage{color}

\theoremstyle{plain}
\newtheorem{thm}{Theorem}[section]

\newtheorem{lem}[thm]{Lemma}
\newtheorem{prop}[thm]{Proposition}

\newtheorem{conj}{Conjecture}

\theoremstyle{definition}

\newtheorem{setup}[thm]{Set-up}
\newtheorem{eg}[thm]{Example}

\theoremstyle{remark}
\newtheorem{rmk}[thm]{Remark}

\def\Z{{\mathbf Z}}

\def\C{{\mathbf C}}

\def\cJ{\mathcal{J}}

\def\cO{\mathcal{O}}

\def\cY{\mathcal{Y}}
\def\cZ{\mathcal{Z}}

\def\Y{\mathcal{Y}}

\def\fra{\mathfrak{a}}
\def\frb{\mathfrak{b}}

\def\frm{\mathfrak{m}}

\def\.{\cdot}
\def\^{\widehat}

\def\({\left(}
\def\){\right)}

\renewcommand{\and}{ \ \ \text{ and } \ \ }

\DeclareMathOperator{\lcm} {lcm}

\DeclareMathOperator{\lct} {lct}
\DeclareMathOperator{\vlct}{vlct}

\begin{document}

\title{On a conjecture of Teissier: the case of log canonical thresholds}

\author{Eva Elduque}
\author{Mircea Musta\c{t}\u{a}}

\address{Department of Mathematics, University of Michigan, 530 Church Street, Ann Arbor, MI 48109, USA}

\email{elduque@umich.edu}
\email{mmustata@umich.edu}

\thanks{The second author was partially supported by NSF grant DMS-1701622.}


\dedicatory{Dedicated to Slava Shokurov, on the~occasion of
his~seventieth~birthday}

\begin{abstract}
For a smooth germ of algebraic variety $(X,0)$ and a hypersurface $(f=0)$ in $X$, with an isolated singularity at $0$,
Teissier \cite{Teissier2} conjectured a lower bound for the Arnold exponent of $f$
in terms of the Arnold exponent of a hyperplane section $f\vert_H$ and the invariant $\theta_0(f)$ of the hypersurface.
By building on an approach due to Loeser \cite{Loeser}, we prove the conjecture in the case of log canonical thresholds.
\end{abstract}

\maketitle

\section{The statement of the conjecture}

Let $X$ be a smooth complex $n$-dimensional algebraic variety, $f\in\cO_X(X)$ a nonzero regular function on $X$ and $P\in X$ a point in the zero-locus of $f$.
We denote by ${\mathfrak m}_P$ the ideal sheaf of regular functions vanishing at $P$ and by
 $J_f$ the Jacobian ideal of $f$: if $x_1,\ldots,x_n$ are algebraic coordinates in an open subset $U$ of $X$ (that is, $x_1,\ldots,x_n$ are regular functions on $U$ such that $dx_1,\ldots,dx_n$ trivialize the cotangent sheaf $\Omega_{U}$), then $J_f$
is generated in $U$ by $\frac{\partial f}{\partial x_1},\ldots,\frac{\partial f}{\partial x_n}$. From now on, we assume that $f$ has an isolated
singularity at $P$, that is, there is an open neighborhood $U$ of $P$ such that $J_f$ does not vanish at any point in $U\smallsetminus\{P\}$.

Teissier introduced and studied in \cite{Teissier1} the invariant $\theta_P(f)$, that can be described by comparing the order of vanishing of $J_f$
with the order of vanishing of ${\mathfrak m}_P$ along the divisorial valuations centered at $P$ (for a precise definition and a brief discussion, see
the beginning of the next section). We also consider
the \emph{Arnold exponent} $\sigma_P(f)$ of $f$ at $P$, that can be defined 
via the asymptotic behavior of integrals over the vanishing cycles of $f$ at $P$, see for example \cite[Section~9]{Kollar}. We consider the following
conjecture of Teissier \cite{Teissier2}:

\begin{conj}\label{conj1}
Let $X$ be a smooth complex $n$-dimensional algebraic variety, $f\in\cO_X(X)$ nonzero, and $P\in X$ a point in the zero-locus of $f$, such that $f$ has an isolated singularity at $P$. 
If $H_1,\ldots,H_{n-1}$ are hypersurfaces in $X$, passing through $P$, such that each $\Lambda_i:=H_1\cap\ldots\cap H_i$ is smooth at $P$,
of dimension $n-i$, and such that $f_i:=f\vert_{\Lambda_i}$ has an isolated singularity at $P$, then 
$$\sigma_P(f)\geq\frac{1}{1+\theta_P(f)}+\frac{1}{1+\theta_P(f_1)}+\ldots+\frac{1}{1+\theta_P(f_{n-1})}.$$
\end{conj}

Note that since ${\rm dim}(\Lambda_{n-1})=1$, we have 
$$\frac{1}{1+\theta_P(f_{n-1})}=\frac{1}{{\rm mult}_P(f_{n-1})}=\sigma_P(f_{n-1}).$$
Therefore the above conjecture is implied by the following 
statement concerning the case of one hypersurface:

\begin{conj}\label{conj2}
Let $X$ be a smooth complex $n$-dimensional algebraic variety, $f\in\cO_X(X)$ nonzero, and $P\in X$ a point in the zero-locus of $f$, such that $f$ has an isolated singularity at $P$. 
If $H$ is a smooth hypersurface in $X$, containing $P$, such that $f\vert_H$ has an isolated singularity at $P$, then
$$\sigma_P(f)\geq \sigma_P(f\vert_H)+\frac{1}{1+\theta_P(f)}.$$
\end{conj}

Loeser proved the above conjecture in \cite{Loeser} when $\theta_P(f)$ is an integer. More generally, he showed that under the assumptions in 
Conjecture~\ref{conj2}, we always have
\begin{equation}\label{eq_Loeser}
\sigma_P(f)\geq \sigma_P(f\vert_H)+\frac{1}{1+\lceil\theta_P(f)\rceil},
\end{equation}
where for a real number $\alpha$, we denote by $\lceil\alpha\rceil$ the smallest integer $\geq\alpha$. 

Our main goal in this note is to prove a version of Conjecture~\ref{conj2} for \emph{log canonical thresholds}. For basic facts about log canonical thresholds,
see \cite[Section~8]{Kollar} or \cite{Mustata}. Recall that if $f$ is a regular function on $X$, having an isolated singularity at $P$, then
the log canonical threshold $\lct_P(f)$ of $f$ at $P$ is related to $\sigma_P(f)$ by
$$\lct_P(f)=\min\{\sigma_P(f),1\}$$
(see \cite[Theorem~9.5]{Kollar}). 
Our goal is to prove the statement corresponding to Conjecture~\ref{conj2} for log canonical thresholds:

\begin{thm}\label{thm_main}
Let $X$ be a smooth complex $n$-dimensional algebraic variety and let $f\in\cO_X(X)$ nonzero and having an isolated singularity at $P$. 
If $H$ is a smooth hypersurface in $X$, containing $P$, such that $f\vert_H\neq 0$, then
$$\lct_P(f)\geq\min\left\{\lct_P(f\vert_H)+\frac{1}{1+\theta_P(f)},1\right\}.$$
\end{thm}

To prove the theorem, we build on Loeser's idea from \cite{Loeser}. Assuming that $f\in\C[x_1,\ldots,x_n]$, $P$ is the origin, and 
$H$ is the hyperplane $x_n=0$, Loeser considers the family 
$$h_t=f(x_1,\ldots,x_{n-1},tx_n)+(1-t)x_n^{\lceil\theta\rceil+1},$$
where $\theta=\theta_P(f)$. The result is then obtained by making use of basic properties of Arnold exponents
(we give more details about this in the next section). Even when $\theta$ is not an integer, we want to consider the
family 
$$h_t=f(x_1,\ldots,x_{n-1},tx_n)+(1-t)x_n^{\theta+1}.$$
In order to make sense of this, we pull back via the cyclic cover $\pi$ 
 given by $(x_1,\ldots,x_n)\to (x_1,\ldots,x_{n-1},x_n^d)$, where $d$ is a suitable positive integer;
 instead of dealing with log canonical thresholds, we have to consider the jumping
 numbers for the multiplier ideals of $h_t\circ\pi$ with respect to the equation $x_n^{d-1}$ defining
 the relative canonical divisor of $\pi$.

\begin{rmk}\label{non_isolated_sing}
We note that unlike in Conjecture~\ref{conj2}, in Theorem~\ref{thm_main} we do not need to require that $f\vert_H$ has isolated singularities
at $P$. On one hand, the log canonical threshold $\lct_P(f\vert_H)$ is defined whenever $f\vert_H$ is nonzero; on the other hand, 
we can always find a smooth hypersurface $H'$ containing $P$ such that $f\vert_{H'}$ has an isolated singularity at $P$ and $\lct_P(f\vert_{H'})\geq\lct_P(f\vert_H)$.
Indeed, after possibly replacing $X$ by an affine open neighborhood of $P$, we may assume that $X$ is affine. We can choose a system
of algebraic coordinates $x_1,\ldots,x_n$ on $X$, centered at $P$, such that $H$ is defined by $x_1$. If we take $H'$ to be defined by a general linear combination
of $x_1,\ldots,x_n$, then $f\vert_{H'}$ has isolated singularities and $\lct_P(f\vert_{H'})\geq\lct_P(f\vert_H)$ by 
the semicontinuity of log canonical thresholds (see \cite[Theorem~3.1]{DemaillyKollar}
or \cite[Theorem~4.9]{Mustata0}).
\end{rmk}

\begin{rmk}\label{holomorphic_case}
One could formulate Theorem~\ref{thm_main} (and similarly Conjecture~\ref{conj2}) in the more general setting when $X$ is a complex manifold and $f$
is a holomorphic function having an isolated singularity at $P$. However, this version follows easily from the algebraic case. Indeed, without any loss
of generality, we may assume that $X\subseteq\C^n$ is an open subset, $P$ is the origin, and $H$ is the hyperplane $(x_n=0)$.
 It is a basic result in singularity theory that if $g$ is a holomorphic function at $0$,
with ${\rm mult}_0(g)\gg 0$ (in fact, it is enough to take ${\rm mult}_0(g)\geq 2+\dim_{\C}(\cO_{\C^n,0}/J_f)$, see
\cite[Corollary~2.24]{GLS}), then $f$ and $f+g$ differ by an analytic change of
coordinates. The same property holds for $f\vert_H$ and $(f+g)\vert_H$ if ${\rm mult}_0(g)\gg 0$ (note that we may assume by 
Remark~\ref{non_isolated_sing} that $f\vert_{H}$ has an isolated singularity at $0$ as well).
We thus have
$$\lct_0(f)=\lct_0(f+g), \quad \lct_0(f\vert_H)=\lct_0\big((f+g)\vert_H),\quad\text{and}\quad \theta_0(f)=\theta_0(f+g).$$
Therefore we may suitably truncate 
$f$ to assume that $f$ is a polynomial. 
\end{rmk}

\begin{rmk}
By building on the approach in this paper, the second author together with Bradley Dirks announced in \cite{DM} a proof of the full Teissier conjecture. As we have already discussed, the proof of Theorem~\ref{thm_main} above relies on the description of the log canonical threshold via multiplier ideals and on various properties of multiplier ideals. The Arnold exponent admits a similar description in terms of Hodge ideals, invariants of singularities that come out
of Saito's theory of mixed Hodge modules (see \cite{MP}). The proof in \cite{DM} follows the approach introduced here, after developing the needed 
results for Hodge ideals (namely, the behavior with respect to finite morphisms and the behavior in families of isolated singularities having constant Milnor number).
\end{rmk}

\subsection{Acknowledgements} We are grateful to Tommaso de Fernex and Alexandru Dimca for useful discussions and to the anonymous
referee for comments on a previous version of the paper.

\section{Proof of the main result}

We begin by reviewing the definition of Teissier's invariant $\theta_P(f)$. Recall that $X$ is a smooth complex algebraic variety, $f\in\cO_X(X)$
is a nonzero regular function, and $P$ is a point in the zero-locus of $f$. We assume that $f$ has an isolated singularity at $P$ and denote by
$J_f$ and ${\mathfrak m}_P$ the Jacobian ideal of $f$, respectively, the ideal defining $P$. 

The description of $\theta_P(f)$ that we use is the following:
\begin{equation}\label{eq1_theta}
\theta_P(f):=\sup_E\frac{{\rm ord}_E(J_f)}{{\rm ord}_E({\mathfrak m}_P)},
\end{equation}
where $E$ varies over the prime divisors on normal complex algebraic varieties $Y$, with a birational morphism $\pi\colon Y\to X$, such that $E$ maps to $P$.
For such $E$, we denote by ${\rm ord}_E(g)$ the coefficient of $E$ in the divisor associated to $g\circ\pi$; for a coherent ideal $J$ on $X$, we denote 
by ${\rm ord}_E(J)$ the minimum of ${\rm ord}_E(g)$, where $g$ runs over a system of generators of $J$ in a neighborhood of $P$.
Note that if $f$ does not have a singularity at $P$, then $J_f=\cO_X$ in a neighborhood of $P$, hence $\theta_P(f)=0$; otherwise, we have
$J_f\subseteq \frm_P$ and thus $\theta_P(f)\geq 1$. 

\begin{eg}\label{dim1}
If $n=1$ and $m={\rm mult}_P(f)\geq 1$, then in a neighborhood of $P$ we can find a coordinate $x_1$ such that $f=gx_1^m$, where 
$g(P)\neq 0$. In this case $J_f$ is generated in a neighborhood of $P$ by $x_1^{m-1}$ and it is clear that $\theta_P(f)=m-1$. 
\end{eg}

It is standard to see that if $\pi\colon Y\to X$ is a proper birational morphism, with $Y$ normal, such that ${\mathfrak m}_P\cdot \cO_Y$ and
$J_f\cdot \cO_Y$ are locally principal ideals (for example, this holds whenever $\pi$ factors through the blow-up of $X$ along ${\mathfrak m}_P\cdot J_f$),
then
\begin{equation}\label{eq2_theta}
\theta_P(f):=\max_{E_i}\frac{{\rm ord}_{E_i}(J_f)}{{\rm ord}_{E_i}({\mathfrak m}_P)},
\end{equation}
where the $E_i$ are the prime divisors on $Y$ that lie in the fiber over $P$. 
Furthermore, 
we can describe $\theta_P(f)$ in terms of the integral closure of the powers of $J_f$ as follows
(for the definition and basic properties of the integral closure of ideals, see for example \cite[Chapter~9.6.A]{Lazarsfeld}). 
If we denote by $\overline{\fra}$ the integral closure of a coherent ideal $\fra$, then
for every positive integers $r$ and $s$, we have 
\begin{equation}\label{eq3_theta}
{\frm_P^r}\subseteq\overline{J_f^s}\,\,\text{in some neighborhood of $P$}\quad\text{if and only if}\quad \frac{r}{s}\geq \theta_P(f)
\end{equation}
(this follows from the characterization of integral closure of ideals in \cite[Proposition~9.6.6]{Lazarsfeld}).

\bigskip

The proof of Theorem~\ref{thm_main} follows the approach in \cite{Loeser}, so we begin by outlining the proof of
(\ref{eq_Loeser}) in \emph{loc. cit.} Recall that if $f$ has an isolated singularity at $P$, then the \emph{Milnor number}
of $f$ at $P$ is
$$\mu_P(f)=\dim_{\C}(\cO_{X,P}/J_f).$$
Arguing as in Remark~\ref{holomorphic_case}, we reduce to the case when $X=\C^n$, 
$P$ is the origin, $f\in\C[x_1,\ldots,x_n]$, and $H$ is defined by $x_n=0$. We put $g(x_1,\ldots,x_{n-1})=f(x_1,\ldots,x_{n-1},0)$.

The key idea is to consider, for a positive integer $m$, the family of polynomials $(h_t)_{t\in\C}$, with
$$h_t(x_1,\ldots,x_n)=f(x_1,\ldots,x_{n-1},tx_n)+(1-t)x_n^m.$$
Note that we have 
$$h_0=g(x_1,\ldots,x_{n-1})+x_n^m\quad\text{and}\quad h_1=f.$$
First, the semicontinuity of the Arnold exponent \cite[Theorem~2.11]{Steenbrink} implies that there is a Zariski open neighborhood $U$ of $0\in\C$ such that
$$\sigma_0(h_t)\geq \sigma_0(h_0)=\sigma_0(g)+\frac{1}{m}\quad\text{for all}\quad t\in U,$$
where the equality follows from the Thom-Sebastiani property of the Arnold exponent (see, for example,
\cite[Example~(8.6)]{Malgrange}).
Second, if $m\geq1+\theta_P(f)$, then there is a Zariski open neighborhood $V$ of $1\in\C$ such that for every $t\in V$, the hypersurface defined
by $h_t$ has isolated singularities at $P$ and the Milnor number at $P$ is constant on $V$:
$$\mu_0(h_t)=\mu_0(f).$$
By a result of Varchenko \cite{Varchenko2}, this implies that the Arnold exponent is constant on this open subset:
$$\sigma_0(h_t)=\sigma_0(h_1)=\sigma_0(f).$$
By taking $t_0\in U\cap V$, we conclude that
$$\sigma_0(f)=\sigma_0(h_1)=\sigma_0(h_{t_0})\geq \sigma_0(g)+\frac{1}{m}.$$
By taking $m=1+\lceil\theta_0(f)\rceil$, we get (\ref{eq_Loeser}). 

We take a similar approach towards the proof of Theorem~\ref{thm_main}, but allowing $m$ to be a rational number.
We use a suitable finite cover in order to make sense of the corresponding log canonical threshold and prove
by ad-hoc methods the semicontinuity and the constancy results needed in this case. In doing this we make use
of various results concerning multiplier ideals. For the definition and basic results on multiplier ideals, we refer to \cite[Chapter~9]{Lazarsfeld}.

In order to justify the definition that follows, let us recall the behavior of the log canonical threshold under finite morphisms.
Recall first that if $f\in\cO_X(X)$ is a nonzero regular function on the smooth variety $X$ vanishing at $P$ and  $\lambda$ is a positive
rational number, then
$\lambda<\lct_P(f)$ if and only if the multiplier ideal $\cJ(X,f^{\lambda})$ is equal to $\cO_X$ around $P$. 
Suppose now that
$\pi\colon Y\to X$ is a finite, surjective morphism between smooth complex algebraic varieties and $K_{Y/X}$ is the relative canonical divisor
(this is the effective divisor locally defined by the determinant of the Jacobian matrix of $\pi$). In this case, it follows from the formula
relating  the multiplier ideals $\cJ(X,f^{\lambda})$ and $\cJ\big(Y, (f\circ\pi)^{\lambda}\big)$ (see \cite[Theorem~9.5.42]{Lazarsfeld})
that
$$
\lambda<\lct_P(f)\quad\text{if and only if}\quad\cO_Y(-K_{Y/X})\subseteq\cJ\big(Y, (f\circ \pi)^{\lambda}\big)\,\,\text{around}\,\,\pi^{-1}(P).
$$

\begin{setup}\label{basic_setup}
We will be interested in the following set-up. Suppose that $X=\C^n$, with $n\geq 2$, and $f\in\C[x_1,\ldots,x_n]$
is nonzero and such that $f(0)=0$. We put $g(x_1,\ldots,x_{n-1})=f(x_1,\ldots,x_{n-1},0)$ and assume $g\neq 0$. We also fix a positive rational number 
$\alpha$. 
Given $t\in \C$, we put
$$h_t:=f(x_1,\ldots,x_{n-1},tx_n)+(1-t)x_n^{\alpha}.$$
We don't attach any concrete meaning to $h_t$, but we define the \emph{virtual log canonical threshold} $\vlct_0(h_t)$
of $h_t$ at $0$, as follows.
We consider a positive integer $d$ such that $d\alpha$ is an integer and the finite surjective
morphism $\pi\colon Y=\C^n\to X$ given by $\pi(u_1,\ldots,u_n)=(u_1,\ldots,u_{n-1},u_n^d)$. 
We denote the standard coordinates on $Y$ by $y_1,\ldots,y_n$.
Note that in this case
the divisor $K_{Y/X}$ is defined by $y_n^{d-1}$. While $h_t$ does not make sense by itself, we may and will consider
$$\widetilde{h}_t=\pi^*(h_t):=f(y_1,\ldots,y_{n-1},ty_n^d)+(1-t)y_n^{d\alpha}\in \C[y_1,\ldots,y_n]=\cO_Y(Y).$$ 
Note that 
$\widetilde{h}_t(0)=0$ and
$\widetilde{h}_t\neq 0$ for all $t$: otherwise, by restricting to 
the hyperplane $y_n=0$ we would get $g=0$, contradicting our assumption.
We put
$$\vlct_0(h_t):=\sup\{\lambda>0\mid y_n^{d-1}\in\cJ(Y,\widetilde{h}_t^{\lambda})\,\,\text{around}\,\,0\}.$$
\end{setup}

\begin{rmk}
Recall that if $\lambda'>\lambda$, then $\cJ(Y,\widetilde{h}_t^{\lambda'})\subseteq \cJ(Y, \widetilde{h}_t^{\lambda})$, with equality if $\lambda'-\lambda$
is small enough (depending on $\lambda$).
This implies that 
 $y_n^{d-1}\in\cJ(Y, \widetilde{h}_t^{\lambda})$ around $0$ if and only if $\lambda<\vlct_0(h_t)$.
\end{rmk}

\begin{rmk}\label{rmk_well_defined}
The definition of $\vlct_0(h_t)$ is independent of $d$. For this, it is enough to show that if instead of $d$ we consider $rd$, 
for a positive integer $r$, then the value of $\vlct_0(h_t)$ does not change. Let $\varphi\colon Z=\C^n\to Y$ be the finite, surjective morphism
given by $\varphi(u_1,\ldots,u_n)=(u_1,\dots,u_{n-1},u_n^r)$ and let us denote by $z_1,\ldots,z_n$ the standard coordinates on $Z$.
 Note that 
$(\pi\circ\varphi)^*(h_t)=\widetilde{h}_t\circ\varphi$. Since $K_{Z/Y}$ is defined by $z_n^{r-1}$, it follows from the behavior of multiplier ideals under finite surjective morphisms
(see \cite[Theorem~9.5.42]{Lazarsfeld})
that 
$$y_n^{d-1}\in \cJ(Y,\widetilde{h}_t^{\lambda})\,\,\text{around}\,\,0\quad\text{if and only if}\quad (y_n^{d-1}\circ\varphi)\cdot z_n^{r-1}=z_n^{rd-1}\in 
\cJ\big(Z, (\widetilde{h}_t\circ\varphi)^{\lambda}\big)
\,\,\text{around}\,\,0.$$
This proves our assertion. In particular, we see that if 
$h_t\in\C[x_1,\ldots,x_n]$ (that is, if
$\alpha$ is an integer or if $t=1$), then $\vlct_0(h_t)=\lct_0(h_t)$. 
\end{rmk}

\begin{rmk}\label{bound_by_1}
Note that for every $t$, we have $\vlct_0(h_t)\leq 1$. Indeed, if $\lambda\geq 1$, then $\cJ(Y,\widetilde{h}_t^{\lambda})\subseteq (\widetilde{h}_t)$.
If $y_n^{d-1}\in \cJ(Y,\widetilde{h}_t^{\lambda})$ around $0$, then there are $u,v\in\C[y_1,\ldots,y_n]$, with $u(0)\neq 0$ such that
\begin{equation}\label{eq_bound_by_1}
u\cdot y_n^{d-1}=\widetilde{h}_t\cdot v.
\end{equation}
Note that $y_n$ does not divide $\widetilde{h}_t$: otherwise, by restricting to the hyperplane $y_n=0$ we would get $g=0$,
contradicting our assumption. We deduce from (\ref{eq_bound_by_1}) that $y_n^{d-1}$ divides $v$ and thus $\widetilde{h}_t$ divides $u$,
contradicting the fact that $u(0)\neq 0$.
\end{rmk}

It is clear that 
$$\vlct_0(h_1)=\lct_0(h_1)=\lct_0(f).$$
The next lemma gives the value at $t=0$.

\begin{lem}\label{lem1}
With the notation is Set-up~\ref{basic_setup}, we have
$$\vlct_0(h_0)=\min\left\{\lct_0(g)+\frac{1}{\alpha},1\right\}.$$
\end{lem}

\begin{proof}
Since $\vlct_0(h_0)\leq 1$ by Remark~\ref{bound_by_1}, in order to prove the equality in the lemma it is enough to show that
for every $\lambda<1$, we have
\begin{equation}\label{eq1_lem1}
\vlct_0(h_0)>\lambda\quad\text{if and only if}\quad \lct_0(g)+\frac{1}{\alpha}>\lambda.
\end{equation}
By definition, we have
$\vlct_0(h_0)>\lambda$ if and only if $y_n^{d-1}\in \cJ(Y,\widetilde{h}_0^{\lambda})$ in a neighborhood of $0$. 
Recall that $\widetilde{h}_0=g(y_1,\ldots,y_{n-1})+y_n^{d\alpha}$. 
We claim that if $\fra$ is the ideal generated by $g$ and $y_n^{d\alpha}$, then 
\begin{equation}\label{eq2_lem1}
\cJ(Y,\fra^{\lambda})=\cJ(Y, \widetilde{h}_0^{\lambda}).
\end{equation}
In order to see this, note first that since $\lambda<1$, it follows from \cite[Proposition~9.2.28]{Lazarsfeld}
that 
$$
\cJ(Y,\fra^{\lambda})=\cJ\big(Y,(c_1g+c_2 y_n^{d\alpha})^{\lambda}\big)
$$
for general elements $c_1, c_2\in\C$. In particular, $c_1$ and $c_2$ are nonzero, and it is clear that we may assume that
$c_1=1$. If we consider the action of $\C^*$ on $\C^n$  given by rescaling the last coordinate, it is clear that $\fra$ is preserved by this action,
hence $\cJ(Y,\fra^{\lambda})$ is preserved as well. Since we can rescale $g+c_2 y_n^{d\alpha}$ to get $g+y_n^{d\alpha}$, we obtain our claim. 

We now use the Summation theorem for multiplier ideals in the form given in \cite{Takagi} (see also \cite{JM}). 
This gives
\begin{equation}\label{eq3_lem1}
\cJ(Y,\fra^{\lambda})=\sum_{\beta+\gamma=\lambda}\cJ(Y,g^{\beta}y_n^{d\alpha\gamma}),
\end{equation}
where the sum is over all nonnegative rational numbers $\beta,\gamma$, with $\beta+\gamma=\lambda$
(note that in the sum there are only finitely many distinct terms).
On the other hand, since $g$ only involves the first $n-1$ variables and since
$\cJ(\C,y_n^{d\alpha\gamma})=(y_n^{\lfloor d\alpha\gamma\rfloor})$, we have by \cite[Proposition~9.5.22]{Lazarsfeld} 
\begin{equation}\label{eq4_lem1}
\cJ(Y,g^{\beta}y_n^{d\alpha\gamma})=\cJ(\C^{n-1},g^{\beta})\cdot \cJ(\C, y_n^{d\alpha\gamma})\cdot \C[y_1,\ldots,y_n]
\end{equation}
$$
=\cJ(\C^{n-1},g^{\beta})y_n^{\lfloor d\alpha\gamma\rfloor}\cdot \C[y_1,\ldots,y_n].
$$
We thus conclude using (\ref{eq2_lem1}), (\ref{eq3_lem1}), and (\ref{eq4_lem1}) that
\begin{equation}\label{eq5_lem1}
\cJ(Y,\widetilde{h}_0^{\lambda})=\sum_{\beta+\gamma=\lambda}\cJ(\C^{n-1},g^{\beta})y_n^{\lfloor d\alpha\gamma\rfloor}\cdot \C[y_1,\ldots,y_n].
\end{equation}

We next claim that this implies that $y_n^{d-1}\in \cJ(Y,\widetilde{h}_0^{\lambda})$ around $0$ if and only if there is $\beta$, with $0\leq\beta\leq\lambda$, 
such that $\beta<\lct_0(g)$ and $\lfloor d\alpha(\lambda-\beta)\rfloor\leq d-1$. Indeed, the fact that this condition is sufficient follows immediately from (\ref{eq5_lem1}).
Conversely, if  $y_n^{d-1}\in \cJ(Y,\widetilde{h}_0^{\lambda})$ around $0$, then it follows from (\ref{eq5_lem1}) that we can write
\begin{equation}\label{eq6_lem1}
q(y_1,\ldots,y_n)\cdot y_n^{d-1}=\sum_{i=1}^ru_i(y_1,\ldots,y_n)p_i(y_1,\ldots,y_{n-1})y_n^{\gamma_i},
\end{equation}
with $q(0)\neq 0$, $u_i\in\C[y_1,\ldots,y_n]$, $p_i\in\cJ(\C^{n-1},g^{\beta_i})$, and $\gamma_i\geq \lfloor d\alpha(\lambda-\beta_i)\rfloor$. After expanding the $u_i$ according to the powers of $y_n$,
it follows that after replacing $r$ by a suitable larger integer, we may and will assume that $u_i=1$ for all $i$. 
By considering the order of vanishing along $y_n=0$, we see that we may assume that $\gamma_i\geq d-1$ for all $i$. Dividing by $y_n^{d-1}$
and evaluating at $0$, we conclude that there is $i$ such that $p_i(0)\neq 0$ and $\gamma_i=d-1$. Therefore $\beta_i<\lct_0(g)$ and 
$d-1\geq \lfloor d\alpha(\lambda-\beta_i)\rfloor$. This completes the proof of our claim.

We thus see that $\lambda<\vlct_0(h_0)$ if and only if there is $\beta$, with $0\leq\beta\leq\lambda$ such that
$\beta<\lct_0(g)$ and $d-1\geq \lfloor d\alpha(\lambda-\beta)\rfloor$. Note that $d-1\geq \lfloor d\alpha(\lambda-\beta)\rfloor$  if and only if 
$d>d\alpha(\lambda-\beta)$, which is equivalent to $\lambda<\beta+\frac{1}{\alpha}$. Therefore
$\lambda<\vlct_0(h_0)$ if and only if $\lambda<\lct_0(g)+\frac{1}{\alpha}$, which completes the proof of the lemma.
\end{proof}

\begin{lem}\label{lem2}
With the notation in Set-up~\ref{basic_setup}, there is a Zariski open neighborhood $U$ of $0\in\C$ such that
$$\vlct_0(h_t)\geq\vlct_0(h_0)\quad\text{for every}\quad t\in U.$$
\end{lem}

\begin{proof}
It is standard to see, using the generic behavior of multiplier ideals in families (see \cite[Theorem~9.5.35]{Lazarsfeld} 
and its proof) and the discreteness of the jumping numbers for the multiplier ideals of a given ideal
(see \cite[Lemma~9.3.21]{Lazarsfeld} and its proof)
that the set $\{\vlct_0(h_t)\mid t\in\C\}$ is finite. 
 It follows that in order to prove the lemma,
it is enough to show that if $\lambda<\vlct_0({h}_0)$, then there is a Zariski open subset $U\subseteq\C$ such that 
$\vlct_0({h}_t)>\lambda$ for all $t\in U$. Note that $\lambda<1$ by Remark~\ref{bound_by_1}.

Given an ideal $I$ in $S:=\C[y_1,\ldots,y_n,t]$ and $t_0\in\C$, we denote by $I_{t_0}$ the restriction of $I$ to the hyperplane
$t=t_0$, identified with $Y$ in the obvious way. 
Let us consider the ideal 
$$\fra=\big(f(y_1,\ldots,y_{n-1},ty_n^d), y_n^{d\alpha}\big)\subseteq S$$
and let $J=\cJ(Y\times\C, \fra^{\lambda})$. If we consider on $S$ the grading given by ${\rm deg}(t)=-d$, ${\rm deg}(y_n)=1$, and 
${\rm deg}(y_i)=0$ for $1\leq i\leq n-1$, then $\fra$ is a homogeneous ideal. This implies that $J$ is homogeneous as well.

Note that if $t_0\neq 0$ and if $s_0\in\C$ is such that $s_0^d=t_0$, then after putting $w=s_0y_n$, we can write
$$f(y_1,\ldots,y_{n-1},t_0y_n^d)+(1-t_0)y_n^{d\alpha}=f(y_1,\ldots,y_{n-1},w^d)+\frac{1-s_0^d}{s_0^{d\alpha}}w^{d\alpha}.$$
Since $\lambda<1$, it follows from  \cite[Proposition~9.2.28]{Lazarsfeld} that for $s_0\in\C$ general, we have
$$\cJ\big(Y, (f(y_1,\ldots,y_{n-1},w^d)+\frac{1-s_0^d}{s_0^{d\alpha}}w^{d\alpha})^{\lambda}\big)
=\cJ(Y, (f(y_1,\ldots,y_{n-1},w^d), w^{d\alpha})^{\lambda}\big).$$
We thus conclude that for $t_0\in\C$ general, we have
$$\cJ(Y, \widetilde{h}_{t_0}^{\lambda})=\cJ(Y,\fra_{t_0}^{\lambda}).$$
Furthermore, we deduce from the theorem regarding the generic behavior of multiplier ideals in families
(see \cite[Theorem~9.5.35]{Lazarsfeld}) that if $t_0\in\C$ is general, then 
$\cJ(Y,\fra_{t_0}^{\lambda})=J_{t_0}$.

Since $\lambda<\vlct_0(h_0)$, it follows that $y_n^{d-1}$ lies in $\cJ(Y,\widetilde{h}_0^{\lambda})$ in a neighborhood of $0$.
Therefore we can find $p\in\C[y_1,\ldots,y_n]$, with $p(0)\neq 0$, such that $p\cdot y_n^{d-1}\in \cJ(Y,\widetilde{h}_0^{\lambda})$.
Note that we have 
$$\cJ(Y,\widetilde{h}_0^{\lambda})\subseteq \cJ(Y,\fra_0^{\lambda})\subseteq J_0,$$
where the first inclusion follows from the fact that $\widetilde{h}_0\in\fra_0$ and the second inclusion
follows from the Restriction Theorem for multiplier ideals (see \cite[Example~9.5.4]{Lazarsfeld}). 
We thus conclude that there are $a_1,\ldots,a_r\in\C[y_1,\ldots,y_n]$
such that
$$q:=p\cdot y_n^{d-1}+\sum_{i=1}^rt^ia_i(y_1,\ldots,y_n)\in J.$$
By only considering the component of $q$ of degree $d-1$ (with respect to the grading that we defined on $S$), 
we may assume that $p\in\C[y_1,\ldots,y_{n-1}]$
and that $a_i=y_n^{id+d-1}b_i(y_1,\ldots,y_{n-1})$ for $1\leq i\leq r$. We can thus write
$$q=y_n^{d-1}\cdot q_1(y_1,\ldots,y_n,t)$$
such that for every $t_0\in\C$, we have $q_1(0,t_0)\neq 0$. 
We then deduce that 
$y_n^{d-1}\in J_{t_0}$ in a neighborhood of $0$. Since for $t_0$ general we have
$J_{t_0}=\cJ(Y,\widetilde{h}_{t_0}^{\lambda})$, we conclude that there
is a Zariski open subset $U$ of $\C$ such that for all $t_0\in U$, we have
$\vlct_0(h_{t_0})>\lambda$. This completes the proof of the lemma.
\end{proof}

\begin{lem}\label{lem2.5}
With the notation in Set-up~\ref{basic_setup}, 
if $m$ is a nonnegative integer such that $m\geq d\cdot \theta_0(f)$, then $y_n^m$ lies in the integral closure
of $J_f\cdot\cO_{Y,0}$. 
\end{lem}

\begin{proof}
Let $\psi\colon \C[x_1,\ldots,x_n]\to \C[y_1,\ldots,y_n]$ be the $\C$-algebra homomorphism corresponding to $\pi$. 
It follows from the characterization of integral closure in terms of divisorial valuations (see \cite[Example~9.6.8]{Lazarsfeld}) that it is enough to show that if
$E$ is a prime divisor on a normal variety $\widetilde{Y}$, with a proper, birational morphism $\widetilde{Y}\to Y$, such that $E$ lies over $0$,
and if ${\rm ord}_E$ is the corresponding valuation, then 
$$m\cdot {\rm ord}_E(y_n)\geq \min_i{\rm ord}_E\big(\psi(\partial f/\partial x_i)\big).$$
Note that there is a prime divisor $F$ on a normal variety $\widetilde{X}$, with a proper, birational morphism 
$\widetilde{X}\to X$, such that ${\rm ord}_E\circ\psi$ is equal to $q\cdot {\rm ord}_F$ on $\C[x_1,\ldots,x_n]$, for some positive integer $q$. 
Indeed, this is equivalent with the fact that if $A\subseteq \C(x_1,\ldots,x_n)$ is the valuation ring corresponding to the restriction of
${\rm ord}_E$, then the residue field of $A$ over $\C$ has transcendence degree $n-1$ (see \cite[Lemma~2.45]{KM}). 
If $B\subseteq \C(y_1,\ldots,y_n)$ is the valuation ring of ${\rm ord}_E$, then $A=B\cap\C(x_1,\ldots,x_n)$
and the assertion follows from
\cite[Chapter~VI.6, Corollary~1]{ZS}.
We note that $F$ lies over $0\in X$. 

By definition of $\theta_0(f)$, we have
$$\min_i{\rm ord}_F(\partial f/\partial x_i)\leq\theta_0(f)\cdot\min_i{\rm ord}_F(x_i)\leq\theta_0(f)\cdot {\rm ord}_F(x_n).$$
On the other hand, we have 
$${\rm ord}_E\big(\psi(\partial f/\partial x_i)\big)=q\cdot {\rm ord}_F(\partial f/\partial x_i)$$
and 
$${\rm ord}_E(y_n)=\frac{1}{d}{\rm ord}_E\big(\psi(x_n)\big)=\frac{q}{d}{\rm ord}_F(x_n).$$
Since $m\geq d\cdot\theta_0(f)$, we conclude that
$$m\cdot {\rm ord}_E(y_n)\geq d\cdot\theta_0(f)\cdot {\rm ord}_E(y_n)=q\cdot\theta_0(f)\cdot {\rm ord}_F(x_n)$$
$$\geq q\cdot \min_i{\rm ord}_F
(\partial f/\partial x_i)=\min_i{\rm ord}_E\big(\psi(\partial f/\partial x_i)\big).$$
This completes the proof.
\end{proof}

The following general result is well-known, but we include a proof for the benefit of the reader.
We consider a smooth morphism $\varphi\colon {\mathcal X}\to T$ between complex algebraic varieties
and a regular function $f$ on ${\mathcal X}$. For $t\in T$, we denote by ${\mathcal X}_t$ the fiber 
$\varphi^{-1}(t)$ and by $f_t$ the restriction $f\vert_{{\mathcal X}_t}$.

\begin{prop}\label{semicont_Milnor_number}
With the above notation, suppose that $s\colon T\to {\mathcal X}$ is such that $\varphi\circ s={\rm id}_T$ 
and $f\vert_{s(T)}=0$. 
Let $t_0\in T$ be such that
$f_{t_0}$  is nonzero and has an isolated singularity\footnote{In this proposition we include in ``isolated singularity" the possibility that the point might be a smooth point.} at $s(t_0)$.
\begin{enumerate}
\item[i)] There is a Zariski open neighborhood $V\subseteq T$ of $t_0$ such that for every $t\in V$, $f_t$ is nonzero,
has an isolated singularity at $s(t)$, and
\begin{equation}\label{eq1_semicont_Milnor_number}
\mu_{s(t)}(f_t)\leq \mu_{s(t_0)}(f_{t_0}).
\end{equation}
\item[ii)] If $V$ is as in i) and such that the inequality in (\ref{eq1_semicont_Milnor_number}) is an equality for all $t\in V$,
then there is an open neighborhood $U$ of $s(V)$ in $\varphi^{-1}(V)$ such that for every $t\in V$, the singular locus of $f_t$
in $U\cap {\mathcal X}_t$ is contained in $\{s(t)\}$. 
\end{enumerate}
\end{prop}

\begin{proof}
For the assertion in i), it is enough to find some nonempty open subset $V\subseteq T$ 
such that (\ref{eq1_semicont_Milnor_number}) holds for all $t\in V$.
Indeed, if $t_0\not\in V$, then arguing by induction on $\dim(T)$, we get
 an open neighborhood $V'$ of $t_0$ in $T\smallsetminus V$ that satisfies the required property for the induced morphism
 $\varphi^{-1}(T\smallsetminus V)\to T\smallsetminus V$. It is then clear that $V\cup V'$ satisfies i).

We have a closed subscheme ${\mathcal Z}$ of ${\mathcal X}$ such that for every $t\in T$, the fiber ${\mathcal Z}_t\hookrightarrow {\mathcal X}_t$
is the subscheme defined by $J_{f_t}$. Let $\psi\colon {\mathcal Z}\to T$ be the morphism induced by $\varphi$. It is easy to see that in order
to prove both i) and ii), we may assume that $s(T)\subseteq {\mathcal Z}$. 

By assumption, $s(t_0)$ is an isolated point in ${\mathcal Z}_{t_0}$. By semicontinuity of fiber dimension, there is an open neighborhood $W$ of
$t_0$ such that $s(t)$ is an isolated point in ${\mathcal Z}_t$ for all $t\in W$. In order to simplify the notation, let us replace ${\mathcal X}$ and $T$
by $\varphi^{_-1}(W)$ and $W$, so that we may assume $W=T$. It is then clear that $s(T)$ is the support of an irreducible component of $\cZ$.
Let $\cZ'\subseteq {\mathcal Z}$ be the union of the other irreducible components of ${\mathcal Z}$ and ${\mathcal Y}$ the
scheme-theoretic closure in ${\mathcal Z}$ of ${\mathcal Z}\smallsetminus {\mathcal Z}'$. Note that the induced morphism
$\psi_0\colon {\mathcal Y}\to T$ is finite, since it induces a bijective map, with inverse $s$, between the corresponding reduced schemes. 
In this case the function
\begin{equation}\label{eq2_semicont_Milnor_number}
T\ni t\to \dim_{\C(t)}\big((\psi_0)_*(\cO_{\Y})\otimes\C(t)\big)
\end{equation}
 is upper semicontinuous, where $\C(t)$ denotes the residue field of the point $t\in T$. Since 
 $$\dim_{\C(t)}\big((\psi_0)_*(\cO_{\cY})\otimes\C(t)\big)\leq\mu_{s(t)}(f_t)$$
 for every $t\in T$, with equality for general $t$ (we use the fact that $\cY$ is a closed subscheme of $\cZ$
 and they are equal around $s(t)$ for general $t\in T$), it follows that there is a nonempty
 open subset $V$ of $T$ such that $\mu_{s(t)}(f_t)\leq\mu_{s(t_0)}(f_{t_0})$ for all $t\in V$.
 As we have seen, this completes the proof of i).
 
 Suppose now that we have equality in (\ref{eq1_semicont_Milnor_number}) for all $t$ in an open subset $V$ of $T$.
Of course, we may and will assume that  $V$ is connected.
In this case we conclude that in fact the function in (\ref{eq2_semicont_Milnor_number}) is constant on $V$ 
and $\mu_{s(t)}(f_t)=\dim_{\C(t)}\big((\psi_0)_*(\cO_{\cY})\otimes\C(t)\big)$ for all $t\in V$. The constancy of the function implies
that $\psi_0^{-1}(V)$ is flat over $V$ (see for example 
 \cite[Theorem~III.9.9]{Hartshorne}). This flatness together with the fact that
 $$\dim_{\C(t)}\left(\cO_{\cZ_t,s(t)}\otimes\C(t)\right)=\dim_{\C(t)}\left(\cO_{\cY,s(t)}\otimes\C(t)\right)\quad\text{for every}\quad t\in V$$
 implies that $\cZ=\cY$ in a neighborhood of $\cY\cap\varphi^{-1}(V)$. This gives the assertion in ii).
\end{proof}

\begin{lem}\label{lem3}
With the notation in Set-up~\ref{basic_setup}, suppose that $0$ is a singular point of $f$ and that
$g=f(x_1,\ldots,x_{n-1},0)$ has an isolated singularity at $0$. If $\alpha\geq\theta_0(f)+1$, then there is a Zariski open
neighborhood $V$ of $1\in\C$ such that
$$\vlct_0(h_t)\leq \vlct_0(h_1)\quad\text{for all}\quad t\in V.$$
\end{lem}

\begin{proof}
 Since $\widetilde{h}_t=f(y_1,\ldots,y_{n-1},ty_n^d)+(1-t)y_n^{d\alpha}$, it follows that
 \begin{equation}\label{eq1_lem3}
 \frac{\partial \widetilde{h}_t}{\partial y_i}=\frac{\partial f}{\partial x_i}(y_1,\ldots,y_{n-1},ty_n^d)\quad\text{for}\quad 1\leq i\leq n-1
 \end{equation}
 and 
 \begin{equation}\label{eq2_lem3}
 \frac{\partial \widetilde{h}_t}{\partial y_n}=\frac{\partial f}{\partial x_n}(y_1,\ldots,y_{n-1},ty_n^d)\cdot dty_n^{d-1}+d\alpha(1-t)y_n^{d\alpha-1}.
 \end{equation}
 
 For $t=1$, we see that $u=(u_1,\ldots,u_n)\in Y$ lies in the zero-locus of $J_{\widetilde{h}_1}$ if and only if $\pi(u)$ lies in the union of the zero-locus of $J_f$
 and of the zero-locus of $J_g$. By assumption, both $f$ and $g$ have isolated singularities at $0$. Since $\pi^{-1}(\{0\})=\{0\}$, we conclude that 
 $\widetilde{h}_1$ has isolated singularities at $0$. We deduce from Proposition~\ref{semicont_Milnor_number}i) 
 that there is a Zariski open neighborhood
 $V$ of $1$ such that $\widetilde{h}_t$ has an isolated singularity at $0$ and $\mu_0(\widetilde{h}_t)\leq\mu_0(\widetilde{h}_1)$
 for every $t\in V$. We may and will assume that $0\not\in V$. 
 The key point is to show that the lower bound on
 $\alpha$ allows us to conclude that
 $\mu_0(\widetilde{h}_t)$ is constant for $t\in V$. 
  
 Given $t\neq 0$, let $s\in\C$ be such that $s^d=t$, and consider the isomorphism $\varphi\colon \C[y_1,\ldots,y_n]\to \C[y_1,\ldots,y_{n-1},w]$,
 given by $\varphi(y_n)=w/s$ and $\varphi(y_i)=y_i$ for $1\leq i\leq n-1$. Note that the ideal $J_s:=\varphi(J_{\widetilde{h}_t})$ is generated by
 $\frac{\partial f}{\partial x_i}(y_1,\ldots,y_{n-1},w^d)$, for $1\leq i\leq n-1$, and
 $$Q_s:=\frac{\partial f}{\partial x_n}(y_1,\ldots,y_{n-1},w^d)\cdot w^{d-1}+\alpha\frac{1-s^d}{s^{d\alpha}}w^{d\alpha-1}.$$
 
 If we put $m=d(\alpha-1)$, then we have by hypothesis $m\geq d\cdot\theta_0(f)$. By Lemma~\ref{lem2.5}, we know that $w^m$ lies in the integral closure of the
 ideal $J$ of $R= \C[y_1,\ldots,y_{n-1},w]_{(y_1,\ldots,y_{n-1},w)}$ generated by
 $\frac{\partial f}{\partial x_i}(y_1,\ldots,y_{n-1},w^d)$, for $1\leq i\leq n$. Note that for every two ideals $\fra$ and $\frb$, we have
 $\overline{\fra}\cdot\overline{\frb}\subseteq \overline{\fra\cdot\frb}$. We thus see that
  \begin{equation}\label{eq3_lem3}
 w^{m+d-1}\in w^{d-1}\cdot\overline{J}\subseteq \overline{w^{d-1}\cdot J}\subseteq \overline{J_1}.
 \end{equation}
 For an $\frm$-primary ideal $\fra$ in $R$, where $\frm$ is the maximal ideal in $R$, we denote by $e(\fra)$ the Hilbert-Samuel multiplicity of $R$
 with respect to $\fra$ (for definition and basic properties of multiplicity, see \cite[Chapter~14]{Matsumura}). Note that by 
 \cite[Theorem~14.13]{Matsumura}, we have $e(\fra)=e({\overline{\fra}})$. Moreover, by \cite[Theorem~14.11]{Matsumura}, if $\fra$ is generated by a system of parameters in $R$
 (which is a regular sequence, since $R$ is Cohen-Macaulay), then $e(\fra)=\dim_{\C}(R/\fra)$. 
 Since $m+d-1=d\alpha-1$, it follows from (\ref{eq3_lem3}) that for every $t$ and $s$ as above, we have
 $Q_s\in\overline{ J_1}$, hence $J_s\subseteq \overline{J_1}$. The
 ideal $J_s\subseteq \frm$ is generated by a system of parameters, hence
 $$\mu_0(\widetilde{h}_t)=\dim_{\C}(R/J_s)=e(J_s)\geq e(\overline{J_1})=e(J_1)=\dim_{\C}(R/J_1)=\mu_0(\widetilde{h}_1).$$
 Since the opposite inequality holds by our choice of $V$, we see 
 that the Milnor number $\mu_0(\widetilde{h}_t)$ is constant for $t\in V$. 
 
 By a result of Varchenko \cite{Varchenko2} (see also \cite[Theorem~2.8]{Steenbrink}),
 the constancy of the Milnor number implies that the spectrum of $\widetilde{h}_t$ at $0$ is constant for $t\in V$. 
 Using the connection between the multiplier ideals of isolated singularities and spectrum, we deduce that in this case,
 for rational number $\lambda\in (0,1)$, the length
 $\dim_{\C}\big(\cO_{Y,0}/\cJ(Y,\widetilde{h}_t^{\lambda})\cO_{Y,0}\big)$ is independent of $t\in V$ (see \cite{Budur}, the main theorem as well as Proposition~2.9
 and its proof).

 Let us fix $\lambda$ as above.
 We also consider the hypersurface $H$ defined by $\widetilde{h}_t$ in $Y\times V$ and the multiplier ideal $\cJ_{\lambda}:=\cJ(Y\times V,\widetilde{h}_t^{\lambda})$. 
 For every $t_0\in V$, we identify $Y\times\{t_0\}$ with $Y$ in the obvious way. 
 Since $\mu_0(\widetilde{h}_t)$ is independent of $t\in V$, it follows 
 from Proposition~\ref{semicont_Milnor_number}ii)
 that there is an open neighborhood $W$ of $\{0\}\times V$ in
 $Y\times V$ such that for every $t_0\in V$, the singular locus of the hypersurface defined by $\widetilde{h}_{t_0}$ in $W\cap \big(Y\times\{t_0\}\big)$ is $\{0\}$. 
 In particular, the singular locus of $H\cap W$ is contained in $\{0\}\times V$. 
 Since $\lambda<1$, this implies that the subscheme $Z_{\lambda}$ of $W$ defined by $\cJ_{\lambda}$ is supported on $\{0\}\times V$. 
 Let $\tau_{\lambda}\colon Z_{\lambda}\to V$ be the finite morphism induced by the projection $Y\times V\to V$. 
 Note that the function
 $$V\ni t\to \dim_{\C(t)}(\tau_{\lambda})_*(\cO_{Z_{\lambda}})\otimes\C(t)$$
 is upper semicontinuous; moreover, it is constant if and only if $Z_{\lambda}$ is flat over $V$ (for the latter assertion, see for example 
 \cite[Theorem~III.9.9]{Hartshorne}).
 On the other hand, it follows from the Restriction Theorem for multiplier ideals (see \cite[Theorem~9.5.1]{Lazarsfeld}) 
 that for every $t_0\in V$, we have
 \begin{equation}\label{eq5_lem3}
 \cJ(Y,\widetilde{h}_{t_0}^{\lambda})\cdot\cO_{Y,0}\subseteq \cJ_{\lambda}\cdot\cO_{Y\times\{t_0\},(0,t_0)}.
 \end{equation} 
 Moreover, this is an equality for general $t_0\in V$ by the behavior of multiplier ideals in families (see \cite[Theorem~9.5.35]{Lazarsfeld}). 
Since  $\dim_{\C}\big(\cO_{Y,0}/\cJ(Y,\widetilde{h}_{t_0}^{\lambda})\cO_{Y,0}\big)$ is independent of $t_0\in V$, we conclude that 
 $Z_{\lambda}$ is flat over $V$ and we have equality in (\ref{eq5_lem3}) for all $t_0\in V$. A consequence of flatness is that
 $Z_{\lambda}$ is the scheme-theoretic closure of $\tau_{\lambda}^{-1}\big(V\smallsetminus\{1\}\big)$ in $Y\times V$ (see \cite[Proposition~III.9.8]{Hartshorne}
 and its proof).
 
After possibly replacing $V$ by a smaller neighborhood of $1$, we may assume that $\vlct_0(h_{t_0})=c$ for all $t_0\in V\smallsetminus\{1\}$
(since there are only finitely many distinct ideals $\cJ_{\lambda}$, with $\lambda\in (0,1)$, and for every such $\lambda$, the set of those $t_0$
with $y_n^{d-1}\in \cJ_{\lambda}\cdot\cO_{Y\times\{t_0\},(0,t_0)}$ is a constructible subset of $V$). We need to show that $\vlct_0(h_1)\geq c$.
For every rational number $\lambda\in (0,c)$, we deduce from (\ref{eq5_lem3}) and our assumption that $y_n^{d-1}\in\cJ_{\lambda}\cdot\cO_{Y\times\{t_0\},(0,t_0)}$ for all $t_0\in V\smallsetminus\{1\}$. 
 This implies that the closed subscheme of $Y\times V$ defined by $y_n^{d-1}$
contains $\tau_{\lambda}^{-1}\big(V\smallsetminus\{1\}\big)$ and thus also contains its scheme-theoretic closure $Z_{\lambda}$. 
Therefore $y_n^{d-1}\in \cJ_{\lambda}\cdot\cO_{Y\times\{1\},(0,1)}=\cJ(Y,\widetilde{h}_{1}^{\lambda})\cdot\cO_{Y,0}$, that is,
$\vlct_0(h_1)>\lambda$. This holds for every $\lambda<c$, hence $\vlct_0(h_1)\geq c$, completing the proof of the lemma. 
\end{proof}

We can now give the proof of our main result.

\begin{proof}[Proof of Theorem~\ref{thm_main}]
Note first that arguing as in Remark~\ref{non_isolated_sing}, we may assume that $g=f\vert_H$ has an isolated singularity at $P$.
Arguing as in Remark~\ref{holomorphic_case}, we may further assume that $X=\C^n$ and $H$ is the hyperplane $(x_n=0)$.
If $f$ is smooth at $0$, then $\lct_P(f)=1$ and $\theta_0(f)=0$, hence the assertion in the theorem is trivial. From now on, we assume that
$f$ is not smooth at $0$, hence $\theta_0(f)>0$.

Let $\alpha=\theta_0(f)+1$ and consider the definitions and notation in Setting~\ref{basic_setup}.
By Lemma~\ref{lem2}, there is a Zariski open neighborhood $U$ of $0\in\C$ such that
\begin{equation}\label{eq1_proof}
\vlct_0(\widetilde{h}_t)\geq \vlct_0(\widetilde{h}_0)\quad\text{for all}\quad t\in U.
\end{equation}
We also know that $\vlct(\widetilde{h}_0)=\min\{\lct_0(g)+\frac{1}{\theta_0(f)+1},1\}$ by Lemma~\ref{lem1}
and that $\vlct_0(\widetilde{h}_1)=\lct_0(f)$. Finally, it follows from Lemma~\ref{lem3} that there is a
Zariski open neighborhood $V$ of $1\in\C$ such that
\begin{equation}\label{eq2_proof}
\vlct_0(\widetilde{h}_t)\leq \vlct_0(\widetilde{h}_1).
\end{equation}
By taking $t\in U\cap V$, we thus conclude using (\ref{eq1_proof}) and (\ref{eq2_proof}) that
$$\lct_0(f)=\vlct_0(\widetilde{h}_1)\geq \vlct_0(\widetilde{h}_t)\geq \vlct_0(\widetilde{h}_0)=\min\left\{\lct_0(g)+\frac{1}{\theta_0(f)+1},1\right\}.$$
This completes the proof of the theorem.
\end{proof}

\section{Examples}
We end the paper with two examples regarding the computation and the combinatorial nature of the invariants involved in Theorem \ref{thm_main} (and Conjecture \ref{conj2}).

For $u=(a_1,\ldots,a_n)\in\Z^n_{\geq 0}$, we write $x^u=x_1^{a_1}\ldots x_n^{a_n}$. Let $\mathfrak{a}=(x^{u_1},\ldots, x^{u_k})\subseteq \C[x_1,\ldots, x_n]$ for some positive integer $k$ and $u_1,\ldots, u_k\in \Z^n_{\geq 0}\backslash \{0\}$, all of them distinct. Let $f_\alpha=\sum_{i=0}^k \alpha_i x^{u_i}$ for $\alpha=(\alpha_1,\ldots,\alpha_k)\in\C^k$. Note that for a general 
$\alpha\in\C^k$, we have 
\begin{equation}\label{eq_example_1}
\lct_0(f_\alpha)=\min\{\lct_0(\mathfrak{a}),1\}
\end{equation}
(see for example \cite[Example~1.10]{Mustata}).
On the other hand, it follows from Howald's formula (see \cite[Example 1.9]{Mustata}) that 
we can compute $\lct_0(\mathfrak{a})$ using \textit{monomial valuations} to get 
\begin{equation}\label{eqn_lct}
\lct_0(\mathfrak{a})=\min_{v\in\Z_{\geq 0}^n\smallsetminus\{0\}} \frac{v_1+\ldots +v_n}
{\min\{\langle u,v\rangle\ \vert\ u\in P(\mathfrak{a})\}},
\end{equation}
where the minimum is over all nonzero $v=(v_1,\ldots,v_n)\in\Z_{\geq 0}^n$, 
we denote by $\langle -,-\rangle$ the usual scalar product in $\Z^n$, and
$P(\mathfrak{a})$ is the Newton polyhedron of $\mathfrak{a}$, defined as
$$
P(\mathfrak{a})= \text{convex hull }\big(\{u\in \Z^n_{\geq 0}\ \vert\ x^u\in\mathfrak{a}\}\big).
$$

In general, if $f\in \C[x_1,\ldots, x_n]$ has an isolated singularity at $0$, we get a lower bound for $\theta_0(f)$ using monomial valuations by the formula
\begin{equation}\label{eqn_bound}
\theta_0(f)\geq\sup_{v\in\Z_{>0}^n} \frac{\min\{\langle u,v\rangle\ \vert\ u\in P\big(\frb(J_f)\big)\}}{\min\{v_1,\ldots,v_n\}},
\end{equation}
where the minimum is over all $v=(v_1,\ldots,v_n)\in\Z_{> 0}^n$ and $\frb(J_f)$ is the smallest monomial ideal containing $J_f$. 
This follows from (\ref{eq1_theta}) and the fact that monomial valuations with center $\{0\}$
correspond to (primitive) $v\in\Z_{>0}^n$; note also 
that for every such valuation $w$, we have $w(J_f)\geq w\big(\frb(J_f)\big)$. 

If, in addition, $J_{f}$ is a monomial ideal, then we have the equality
\begin{equation}\label{eqn_theta}
\theta_0(f)=\max_{v\in\Z_{>0}^n} \frac{\min\{\langle u,v\rangle\ \vert\ u\in P(J_f)\}}{\min\{v_1,\ldots,v_n\}}.
\end{equation}
Indeed, this follows from (\ref{eq2_theta}) and the fact that the normalized blow-up at the monomial ideal $\frm_0\cdot J_{f}$ is a map of toric varieties, hence
 the corresponding prime divisors over $0$ are torus invariant and give rise to monomial valuations corresponding to primitive elements in $\Z_{>0}^n$.

We say that a divisorial valuation $w={\rm ord}_E$ of $\C(x_1,\ldots,x_n)$ computes the log canonical threshold ${\rm lct}_0(f)$ of $f$ if $0$ lies in the closure of the image of $E$
and if $w$ achieves the minimum in the definition of $\lct_0(f)$ via divisorial valuations. If $H$ is a hyperplane containing $0$, then it is a consequence of Inversion of Adjunction (see \cite[Corollaries~9.5.11, 9.5.17]{Lazarsfeld}) that
$$\lct_0(f\vert_H)={\rm lct}_0\big((\C^n,H),f\big).$$
We thus say that a divisorial valuation $w$ of $\C(x_1,\ldots,x_n)$ computes $\lct_0(f\vert_H)$ if it computes ${\rm lct}_0\big((\C^n,H),f\big)$.

We first give an example in which the inequality in Theorem~\ref{thm_main} is an equality, we have a unique divisorial valuation that computes
$\lct_0(f)$, this also computes $\lct_0(f\vert_H)$, but does not compute $\theta_0(f)$.

\begin{eg}
Let $f(x_1,\ldots,x_n)=x_1^{a_1}+\ldots+x_n^{a_n}$, with $a_1\geq a_2\geq\ldots\geq a_n\geq 2$, $a_1>a_n$ and $n\geq  2$, and let $H$ be the hyperplane defined by $x_1=0$. Note that both $f$ and $f\vert_{H}$ have isolated singularities at the origin. Suppose that $\frac{1}{a_1}+\ldots+\frac{1}{a_n}\leq 1$. Using (\ref{eq_example_1}) and (\ref{eqn_lct}), we see that
$$
\lct_0(f)=\lct_0\big((x_1^{a_1},\ldots,x_n^{a_n})\big)=\frac{1}{a_1}+\ldots+\frac{1}{a_n}.
$$
Moreover, the minimum in (\ref{eqn_lct}) is only achieved when $v$ is a multiple of 
$$\lcm(a_1,\ldots,a_n)\left(\frac{1}{a_1},\ldots,\frac{1}{a_n}\right);$$
the corresponding divisorial valuation ${\rm ord}_E$ is the unique one computing $\lct_0(f)$.
Similarly, we see that $\lct_0(f\vert_H)=\frac{1}{a_2}+\ldots+\frac{1}{a_n}$ and it is straightforward to check that
${\rm ord}_E$ also computes $\lct_0(f\vert_H)$.

Using (\ref{eqn_theta}), one can check that $\theta_0(f)=a_1-1$, which is achieved, for example, when $v=(1,\ell,\ldots,\ell)$ for large $\ell$. However, 
we note that 
since $a_1\neq a_n$, if we plug in \linebreak $v=\lcm(a_1,\ldots,a_n)\left(\frac{1}{a_1},\ldots,\frac{1}{a_n}\right)$ to the right-hand side of (\ref{eqn_theta}), we do not achieve the maximum. Hence, we cannot use the same valuation to compute both $\lct_0(f)$ and $\theta_0(f)$.

Finally, note that in this example, the inequality in Theorem \ref{thm_main} is an equality.
\end{eg}

Our second example deals with the case when $J_{f_\alpha}$ is not a monomial ideal. It shows that in this case, even if $\alpha$ is general, we cannot use monomial valuations to bound $\theta_0(f_\alpha)$ via (\ref{eqn_bound}) well enough to prove Theorem \ref{thm_main} combinatorially.

\begin{eg}
Let $f_\alpha=\alpha_1x^7+\alpha_2y^2+\alpha_3x^5y\in \C[x,y]$, with $\alpha=(\alpha_1,\alpha_2,\alpha_3)\in(\C^*)^3$ general. Let $H$ be
a line with equation of the form $x=\beta y$, for any $\beta\in\C$. Using (\ref{eq_example_1})
and (\ref{eqn_lct}), we compute
$$
\lct_0(f_\alpha)=\frac{9}{14}\quad\text{and}\quad
\lct_0(f_\alpha\vert_H)=\frac{1}{2}.
$$
We deduce from Theorem \ref{thm_main}  that $\frac{1}{1+\theta_0(f_\alpha)}\leq \frac{1}{7}$, or equivalently, $\theta_0(f_\alpha)\geq 6$. However, equation (\ref{eqn_bound}) only gives us the weaker bound
$$
\theta_0(f_{\alpha})\geq \sup_{v\in\Z^2_{>0}} \frac{\min\{\langle u,v\rangle\ \vert\ u\in P\big(\frb(J_{f_\alpha})\big)\}}{\min\{v_1,v_2\}}=5.
$$
\end{eg}

\end{document}